\documentclass[11pt,a4paper,english]{article}

\usepackage[T1]{fontenc}
\usepackage[utf8]{inputenc}
\usepackage{babel}
\usepackage[final]{microtype} % Improved typography
\usepackage{lmodern}

\usepackage[a4paper, hmargin={2.7cm,2.7cm},vmargin={3.3cm,3.3cm}]{geometry}
\usepackage{fancyhdr, lastpage}

\usepackage{mathtools} % Loads amsmath automatically, adds delimiter tools
\usepackage{amssymb, amsfonts, amsthm, amsxtra, bm}
\usepackage{mathrsfs}  % For \mathscr
\usepackage{dsfont}    % often used for indicator functions

\usepackage{graphicx}
\graphicspath{{Figures/}}
\usepackage{tikz}
\usepackage{pgfplots}
\usepgfplotslibrary{fillbetween}
\pgfplotsset{compat=newest}

\definecolor{sienna}{RGB}{160,82,45}
\definecolor{violet}{RGB}{238,130,238}
\definecolor{lightorange}{RGB}{253,208,162}

\usepackage{caption}
\usepackage[margin=0.2cm]{subcaption}
\usepackage{float}

\usepackage{enumitem} % Modern replacement for paralist
\setlist{topsep=0.4em, partopsep=0.2em, itemsep=0.1em, parsep=0.05em}
\usepackage{url}
\usepackage{siunitx}
\usepackage{etoolbox} % For patching commands
\usepackage{aliascnt} % Preserve theorem-like environment types for cleveref
\usepackage{cite}
\usepackage[bf,compact,small]{titlesec}
\titlespacing*{\section}{0pt}{14pt}{4pt}
\titlespacing*{\subsection}{0pt}{8pt}{3pt}

\makeatletter
\patchcmd{\ttlh@hang}{\parindent\z@}{\parindent\z@\leavevmode}{}{}
\patchcmd{\ttlh@hang}{\noindent}{}{}{}
\makeatother

\pgfplotsset{every axis/.append style={
    axis x line=middle,
    axis y line=middle,
    axis line style={->},
    xlabel style={at={(ticklabel* cs:1)},anchor=north west},
}}

\DeclarePairedDelimiter{\itvcc}{\lbrack}{\rbrack}      % Closed interval [a,b]
\DeclarePairedDelimiter{\itvoc}{\lparen}{\rbrack}      % Open-Closed (a,b]
\DeclarePairedDelimiter{\itvco}{\lbrack}{\rparen}      % Closed-Open [a,b)
\DeclarePairedDelimiter{\itvoo}{\lparen}{\rparen}      % Open (a,b)
\DeclarePairedDelimiter{\abs}{\lvert}{\rvert}          % Absolute value
\DeclarePairedDelimiter{\floor}{\lfloor}{\rfloor}      % Floor

\DeclarePairedDelimiterX\set[1]\lbrace\rbrace{%
   #1}

\numberwithin{equation}{section}
\allowdisplaybreaks[4]

\newtheorem{theorem}{Theorem}[section]

\newaliascnt{lemma}{theorem}
\newtheorem{lemma}[lemma]{Lemma}
\aliascntresetthe{lemma}
\newaliascnt{proposition}{theorem}
\newtheorem{proposition}[proposition]{Proposition}
\aliascntresetthe{proposition}
\newaliascnt{corollary}{theorem}
\newtheorem{corollary}[corollary]{Corollary}
\aliascntresetthe{corollary}
\theoremstyle{definition}

\theoremstyle{remark}

\newtheorem{remark}{Remark}

\DeclareMathOperator{\exponential}{e}

\DeclareMathOperator{\sgn}{sgn}
\DeclareMathOperator{\sinc}{sinc}

\newcommand{\R}{\mathbb{R}}

\newcommand{\Z}{\mathbb{Z}}
\newcommand{\N}{\mathbb{N}}

\newcommand{\gaborG}[1]{\mathcal{G}(#1,a,b)} % Gabor system shorthand
\newcommand{\sfrac}[1]{F{(#1)}} % Signed fractional part
\newcommand{\round}[1]{R{(#1)}} % Rounding to nearest integer
\newcommand{\myexp}[1]{\exponential^{#1}}

\renewcommand{\Re}{\operatorname{Re}}

\makeatletter
\def\maketimestamp{\count255=\time
\divide\count255 by 60\relax
\edef\thetime{\the\count255:}%
\multiply\count255 by-60\relax
\advance\count255 by\time
\edef\thetime{\thetime\ifnum\count255<10 0\fi\the\count255}
\edef\thedate{\number\day-\ifcase\month\or Jan\or Feb\or Mar\or
             Apr\or May\or Jun\or Jul\or Aug\or Sep\or Oct\or
             Nov\or Dec\fi-\number\year}
\def\timstamp{\hbox to\hsize{\tt\hfil\thedate\hfil\thetime\hfil}}}
\makeatother
\maketimestamp

\fancypagestyle{plain}{%
\fancyhf{} % clear all header and footer fields
\fancyfoot[C]{\small \thepage{} of \pageref{LastPage}}}

\makeatletter
\def\blfootnote{\xdef\@thefnmark{}\@footnotetext} 
\if@titlepage
  \renewenvironment{abstract}{%
      \titlepage \null\vfil
      \@beginparpenalty\@lowpenalty
      \begin{center}\bfseries \abstractname \@endparpenalty\@M \end{center}}%
     {\par\vfil\null\endtitlepage}
\else
  \renewenvironment{abstract}{%
      \small
      \list{}{\setlength{\leftmargin}{3em}\setlength{\rightmargin}{3em}}
      \item[\textbf{\abstractname:}]}
      {\endlist}
\fi
\makeatother

\usepackage{hyperref}
\hypersetup{
    colorlinks=true,
    linkcolor={red!50!black},
    citecolor={blue!50!black},
    urlcolor={blue!80!black},
    pdfview={FitH},
    pdfauthor={Marzieh Hasannasab and Jakob Lemvig},
    pdftitle={The zero set of the Zak transform of B-splines with applications to Gabor frames},
}

\usepackage[noabbrev,nameinlink,capitalize]{cleveref}

\begin{document}

\title{The zero set of the Zak transform of B-splines with applications to Gabor frames}
\author{Marzieh Hasannasab\footnote{Technical University of Denmark. E-mail: \protect\url{mhas@dtu.dk}}\phantom{$\ast$}, Jakob Lemvig\footnote{Technical University of Denmark. E-mail: \protect\url{jakle@dtu.dk}}}
\date{\today}


\maketitle
\thispagestyle{plain}

\begin{abstract}
    We study the zero sets of the Zak transform $Z_\lambda B_n(x,\nu)$ of B-splines for $\lambda > 0$. Specifically, we provide a full characterization of the zero set for the hat spline for all positive values of the parameter $\lambda$ and for higher order B-splines when $\lambda > 1$. Finally, we apply these results to establish the frame property of integer-oversampled Gabor systems generated by B-splines. In particular, we characterize the frame property for integer-oversampled Gabor systems generated by the hat splines.
\end{abstract}

\section{Introduction}
\label{sec:intro}

The Zak transform of a function $f \in L^2(\R)$ is defined by:
\begin{equation}
    \label{eq:zak-transform-def}
    Z_\lambda f(x, \nu) = \sqrt{\lambda} \sum_{k \in \Z} f(\lambda(x-k)) \myexp{2 \pi i k \nu}.
\end{equation}
with convergence in $L^2_{\mathrm{loc}}(\R^2)$. The values of $Z_\lambda f$ on the fundamental domain $\itvco{0,1}^2$ uniquely determine its values  on $\R^2$, and the Zak transform is a unitary operator from $L^2(\R)$ to $L^2(\itvco{0,1}^2)$, see, e.g., \cite{MR947891,JanssenZak2003,MR1843717}.

It is well known that if the Zak transform $Z_\lambda f$ is a continuous function, then it must possess at least one zero in the fundamental domain $\itvco{0,1}^2$ \cite{MR1843717,MR947891}. For certain function classes, such as totally positive functions \cite{MR3393698,MR3692123,MR3218799}, this is the \emph{unique} zero in $\itvco{0,1}^2$, located at $(x^*, 1/2)$ for some $x^* \in \itvco{0,1}$. Because totally positive functions are scale-invariant, it suffices to study the zeros of $Z_\lambda f$ for a single value of $\lambda$; specifically, if $f_\lambda(\cdot) \coloneqq f(\lambda \cdot)$, then $Z_\lambda f = Z_1 f_\lambda$. Consequently, considerable attention has been given to the study of $Z_1 f$. For instance Janssen~\cite{JanssenZak2003} showed that for super convex functions $f$, the Zak transform $Z_1 f$ has a unique zero at $(1/2, 1/2)$. 

In this work, we focus on the Zak transform of the centered B-splines $B_n$ of order $n \in \N$, which are defined recursively by convolution as $B_1 = \chi_{\itvco{-1/2, 1/2}}$ and $B_{n+1} = B_{n} \ast B_{1}$. Because B-splines with integer knots lack the scale invariance of totally positive functions, the aforementioned simplification does not apply. Thus, one must study $Z_\lambda B_n$ across all scaling parameters $\lambda > 0$, a regime that remains largely unexplored in the literature. As we will demonstrate, the behavior and geometric structure of the zero set of $Z_\lambda B_n$ depend delicately on the parameter $\lambda$. 
Historically, Schoenberg investigated $Z_1 B_n$ under the framework of exponential Euler splines \cite{MR420078}, and it was subsequently established that $Z_1 B_n$ possesses a unique zero at $(1/2, 1/2)$ for all $n \ge 2$ \cite{MR1113840}. Further structural and geometric properties of $Z_1 B_n$ was studied in, e.g., \cite{MR1113840,MR493059}. However, the methods used in these works are not applicable to the case of general $\lambda$ as they rely on the invariance of the integer knots under the scaling by $\lambda$. The second named author showed with K.~Nielsen in~\cite{MR3572909} that the zero set of $Z_\lambda B_n$ contains a number of horizontal lines when $\lambda < 1$ (see Proposition \ref{thm:zeros-zak-bspline} below) and used this zero set to disprove the so-called frame set conjecture for B-splines \cite{Groechenigmystery2014}.
Further results on the frame set of B-splines, including positive frame regions and additional non-frame obstructions, can be found in \cite{ChristensenKimKim2015, Grochenig2015, atindehou2023frame, lemvig2026new, GhoshSelvan2025}.
% =======
% Historically, Schoenberg investigated $Z_1 B_n$ under the framework of exponential Euler splines \cite{MR420078}, and it was subsequently established that $Z_1 B_n$ possesses a unique zero at $(1/2, 1/2)$ for all $n \ge 2$ \cite{MR1113840}. Further structural and geometric properties of $Z_1 B_n$ were studied in, e.g., \cite{jetter1991schoenberg,MR493059}. However, the methods used in these works are not applicable to the case of general $\lambda$ as they rely on the invariance of the integer knots under scaling by $\lambda$. The second named author showed with K.~Nielsen in~\cite{MR3572909} that the zero set of $Z_\lambda B_n$ contains a number of horizontal lines when $\lambda < 1$ (see Proposition \ref{thm:zeros-zak-bspline} below) and used this zero set to disprove the so-called frame set conjecture for B-splines \cite{Groechenigmystery2014}.

The main contribution of this paper is a comprehensive analysis of the zero set of $Z_\lambda B_n$ for both small and large scaling parameters. We show that for all $n \ge 2$, the zero set of $Z_\lambda B_n$ consists of a single point at $(1/2, 1/2)$ whenever $1 \le \lambda < n$. Conversely, for small scaling values ($\lambda < 1$), the zero set exhibits a significantly more intricate structure. In this regime, the fixed point at $(1/2, 1/2)$ and the horizontal lines identified in \cite{MR3572909} do not exhaust the zero set. We demonstrate that additional zeros emerge along the line $x=1/2$, arising as the roots of a specific class of palindromic polynomials. We provide a complete characterization of the zero set for the linear hat spline $B_2$ for all $\lambda > 0$ and for higher-order B-splines when $\lambda \ge 1$.

Characterizing the zero set of the Zak transform for all values of $\lambda$ is of primary interest because it is closely linked to the frame properties of Gabor systems generated by B-splines. For integer-oversampled Gabor systems 
\[ 
    \gaborG{B_n} = \set{\myexp{2\pi i b k\cdot} B_n(\cdot - am)}_{k,m \in \Z}, 
\] 
where $ab = 1/q$ for some integer $q \ge 2$, the Zibulski-Zeevi characterization \cite{zibulski1997analysis} states that $\gaborG{B_n}$ is \emph{not} a frame if and only if the Zak transform $Z_{1/b} B_n(\cdot, \nu_0)$ exhibits $q$ zeros spaced $1/q$ apart on the same horizontal line $\nu = \nu_0$ for some $\nu_0 \in \itvco{0,1}$. That is, the system fails to be a frame if and only if
\begin{equation}
    \label{eq:ZZ-characterization-q-zeros}
    Z_{1/b} B_n(x_0+\ell/q, \nu_0) = 0 \quad \text{for $\ell=0,1,\ldots,q-1$}
\end{equation}
for some $(x_0, \nu_0) \in \itvco{0,1}^2$.

In Section~\ref{sec:preliminaries}, we introduce the necessary preliminaries and prove a key aliasing lemma. In Section~\ref{sec:zeros-higher-order-bspline-small-b}, we study the Zak transform $Z_\lambda B_n$, $n\ge2$, in the regime $\lambda>1$. Section~\ref{sec:zeros-boundary-lines} is devoted to the zeros of $Z_\lambda B_n$ along certain boundary lines. In Section~\ref{sec:zeros-set-Bn}, we combine these results to obtain a complete description of the zero set of $Z_\lambda B_n$ for $\lambda>1$, and a complete characterization of the zero set of $Z_\lambda B_2$ for $\lambda<1$. In Section~\ref{sec:palindromic-polynomials}, we discuss additional zeros of $Z_\lambda B_n(1/2,\cdot)$ that arise from the roots of palindromic polynomials. Finally, in Section~\ref{sec:frame-property-gabor-systems}, we apply the zero-set results to characterize the frame property of integer-oversampled Gabor systems generated by B-splines.

\paragraph{Added note.} During the final editing of this manuscript, Alexander Stangl and Christina Frederick published a paper on the Zak transform of the hat spline with overlapping results on arXiv \cite{stangl2026frameset}. In particular, the paper also characterized the frame property of the integer oversampled Gabor systems generated by the hat spline.  

\section{Preliminaries}
\label{sec:preliminaries}

The Zak transform of a function $Z_\lambda f(x, \nu)$ is periodic in the variable $\nu$ with period $1$, and quasi-periodic in the variable $x$ with phase factor $z = \myexp{2\pi i \nu}$:
\begin{equation}
    \label{eq:zak-quasi-periodic}
    Z_\lambda f(x+1, \nu) = \myexp{2\pi i \nu} Z_\lambda f(x, \nu).
\end{equation}
Hence, we can identify the Zak transform with its restriction to the fundamental domain $\itvco{0,1}^2$. 

\subsection{Symmetries of the Zak transform}
\label{sec:symmetries-zak-transform}

Assume for the remainder of this section that $f$ is a real and even function whose Zak transform is continuous, e.g., the centered B-splines $B_n$ for $n \ge 2$. Then, the Zak transform $Z_\lambda f$ has the following symmetries:
\begin{equation}
    \label{eq:Zak-symmetry}
    Z_\lambda f(x, \nu) = Z_\lambda f(-x, -\nu) = \overline{Z_\lambda f(x, -\nu)}
\end{equation}
It follows, using also the quasi-periodicity, that $Z_\lambda f(x+1/2, 1/2) = -Z_\lambda f(1/2-x, 1/2)$ for all $x \in \R$. In particular, $Z_\lambda f(1/2, 1/2) = -Z_\lambda f(1/2, 1/2)$ which shows that $Z_\lambda f(1/2, 1/2)=0$ for all $\lambda>0$.

From \eqref{eq:Zak-symmetry}, it also follows that the fundamental domain can be reduced to $\itvcc{0,1/2}^2$ in the sense that the values of $Z_\lambda f$ on $\itvcc{0,1/2}^2$ determine the values on $\R^2$. It will be helpful to distinguish between zeros of $Z_\lambda B_n$ that are located in the interior of the reduced fundamental domain $\itvcc{0,1/2}^2$ and those that are located on the boundary of $\itvcc{0,1/2}^2$:
\begin{enumerate}[label=(\Roman*)]
    \item interior points: $(x,\nu) \in \itvoo{0,1/2}^2$, and
    \item boundary points: $(x,\nu) \in \itvcc{0,1/2}^2 \setminus \itvoo{0,1/2}^2$.
\end{enumerate}

On the boundary of the reduced fundamental domain $\itvcc{0,1/2}^2$, the Zak transform $Z_\lambda f$ is real-valued (up to a phase factor) as shown in the following lemma. 
\begin{lemma}
    \label{lem:real-valued-Zak_along-boundary}
    Let $\lambda > 0$, and let $f$ be a real and even function whose Zak transform is well-defined. The scalar functions $Z_\lambda f(x, 0)$, $Z_\lambda f(x, 1/2)$, $Z_\lambda f(0, \nu)$ and $\myexp{-\pi i \nu} Z_\lambda f(1/2, \nu)$ are real-valued for all $x,\nu \in \R$.
\end{lemma}

\begin{proof}
    By quasi-periodicity we have $Z_\lambda f(1/2, \nu) = \myexp{2\pi i \nu} Z_\lambda f(-1/2, \nu)$. Using the symmetry $Z_\lambda f(-x, \nu) = \overline{Z_\lambda f(x, \nu)}$ we then get $Z_\lambda f(1/2, \nu) = \myexp{2\pi i \nu} \overline{Z_\lambda f(1/2, \nu)}$. Hence $\myexp{-\pi i \nu} Z_\lambda f(1/2, \nu) = \overline{\myexp{-\pi i \nu} Z_\lambda f(1/2, \nu)}$ which shows that
    $\myexp{-\pi i \nu} Z_\lambda f(1/2, \nu)$ is real-valued for all $\nu$. The other cases follow similarly.
\end{proof}

We will also split the analysis into $\lambda > 1$ and $\lambda < 1$ since the zero set of $Z_\lambda B_n$ exhibits different behavior in these two regimes. Recall that the case $\lambda = 1$ is classical and well understood: $Z_1 B_n$ has a unique zero at $(1/2, 1/2)$ for all $n \ge 2$ \cite{MR1113840}.

\subsection{Aliasing lemma}
\label{sec:aliasing-lemma}

For $x\in \R$ we let $\round{x}$ denote the round function to the nearest integer, i.e., $\round{x}=\floor{x+\frac12}$, and let $\sfrac{x} = x-\round{x} \in \itvco{-\frac12,\frac12}$ denote the (signed) fractional part of $x$. Our results will often involve the parity of $\round{x}$ and the distance of $x$ to the nearest integer $\abs{\sfrac{x}} = \operatorname{dist}(x,\Z)$. This is also the case of the first useful aliasing lemma. 

\begin{lemma}
    \label{lem:aliasing}
    Let $n\ge2$, $b>0$, and $\nu\in(0,1)$. Suppose that $M\in\Z$
    satisfies $M\nu=\ell\in\Z$, and write $b=M+F$ and $d=\abs{F}$.
    If $d>0$, then
    \begin{equation}
		\label{eq:aliasing-lemma}
    Z_{1/b}B_n(x,\nu)
    =
    C
    \begin{cases}
        Z_{1/d}B_n(x,\nu),
        & nM \text{ even},\\
        \myexp{\pi i\nu}
        Z_{1/d}B_n\left(x-\frac12,\nu\right),
        & nM \text{ odd},
    \end{cases}		
	\end{equation}
    where
    \[
        C
        =
        (-1)^{n\ell}\sgn(F)^n
        \left(\frac{d}{b}\right)^{n-\frac12}
        \neq0.
    \]
    If $d=0$, then $Z_{1/b}B_n(\cdot,\nu)\equiv0$.
\end{lemma}

\begin{proof}
    By Poisson summation and
    \[
        \widehat{B_n}(\xi)
        =
        \left(\frac{\sin(\pi\xi)}{\pi\xi}\right)^n,
    \]
    we have
    \[
        Z_{1/b}B_n(x,\nu)
        =
        \sqrt b
        \sum_{m\in\Z}
        \widehat{B_n}\bigl(b(m+\nu)\bigr)
        \myexp{2\pi i(m+\nu)x}.
    \]
    Since $b=M+F$ and $M\nu=\ell$,
    \[
        b(m+\nu)
        =
        Mm+\ell+F(m+\nu),
    \]
    and therefore
    \[
        \sin\bigl(\pi b(m+\nu)\bigr)
        =
        (-1)^{Mm+\ell}
        \sin\bigl(\pi F(m+\nu)\bigr).
    \]
    Writing $d=\abs{F}$ gives
    \[
        Z_{1/b}B_n(x,\nu)
        =
        C\sum_{m\in\Z}
        (-1)^{nMm}
        \sqrt d\,
        \widehat{B_n}\bigl(d(m+\nu)\bigr)
        \myexp{2\pi i(m+\nu)x}.
    \]
    If $nM$ is even, the factor $(-1)^{nMm}$ is identically one.
    If $nM$ is odd, then $(-1)^{nMm}=(-1)^m$, and
    \[
        (-1)^m \myexp{2\pi i(m+\nu)x}
        =
        \myexp{\pi i\nu}
        \myexp{2\pi i(m+\nu)(x-\frac12)}.
    \]
    This proves the assertion for $d>0$.

    If $d=0$, then $b=M$ and $M\nu=\ell\in\Z$. Since $\nu\in(0,1)$,
    $M(m+\nu)=Mm+\ell$ is a nonzero integer for every $m\in\Z$.
    Hence all terms in the Poisson representation vanish.
\end{proof}
\section{The Zak transform $Z_\lambda B_n$ for $\lambda > 1$}
\label{sec:zeros-higher-order-bspline-small-b}

In this section we study the zero set of $Z_\lambda B_n$ on the domain $\itvco{0,1}^2$ for $\lambda > 1$. We first consider the case $\lambda \ge n$, which is the simplest ``Painless'' case, where the zero set of $Z_\lambda B_n$ is a vertical strip centered around $x = 1/2$.

We then, in~\Cref{thm:zeros-zak-bspline-medium-lambda}, consider the more delicate and interesting case $1 < \lambda < n$.

\begin{proposition}
	\label{thm:zeros-zak-bspline-large-lambda}
	Let $n\geq 2$ and $\lambda \ge n$. Then $Z_\lambda B_n(x,\nu) = 0$ for $(x,\nu) \in \itvco{0,1}^2$ if and only if $x \in \itvcc{n/(2\lambda), 1-n/(2\lambda)}$ and $\nu \in \itvco{0,1}$. 
\end{proposition}
\begin{proof}
	This is the painless case, and the assertion follows directly from the support of the B-spline $B_n$. Recall that
\[
B_n(\lambda(x-k))\neq0
\quad\Leftrightarrow\quad
|x-k|<\frac{n}{2\lambda}.
\]
For $\lambda\ge n$, we have $n/(2\lambda)\le1/2$. Hence, for $x\in\itvco{0,1}$, all terms in the Zak sum vanish exactly when
\[
\frac{n}{2\lambda}\le x\le 1-\frac{n}{2\lambda}.
\]
Outside this interval exactly one term is nonzero, so $Z_\lambda B_n(x,\nu)\neq0$ for every $\nu\in\itvco{0,1}$. Therefore
\[
Z_\lambda B_n(x,\nu)=0 \quad
\Leftrightarrow \quad
x\in\itvcc*{\frac{n}{2\lambda},1-\frac{n}{2\lambda}}.
\]
\end{proof}

\Cref{thm:zeros-zak-bspline-large-lambda} shows that for the critical value $\lambda = n$, the Zak transform $Z_\lambda B_n$ vanishes exactly along the line $x = 1/2$. As $\lambda$ decreases below $n$, this zero set disappears: for $1<\lambda<n$, the Zak transform $Z_\lambda B_n$ has no zeros in the domain $\itvco{0,1}\times \itvoo{0,1/2}$ as we show next. In this sense, the domain is essentially maximal, since $Z_\lambda B_n(1/2, 1/2) = 0$ for all $\lambda > 0$.

\begin{theorem}
	\label{thm:zeros-zak-bspline-medium-lambda}
	Let $n\geq 2$ and $1<\lambda<n$. For every $\nu\in(0,1/2)$, we 
    define
	\begin{equation}\label{def:F_n}
	F_{n}(x)
	=
	\lambda^{-\frac{1}{2}} Z_\lambda B_n(x,\nu).
	\end{equation}
	Then $F_{n}(\cdot)$ is nonzero on $\mathbb R$. In particular, the argument admits a continuous branch $\theta_{n}(x)=\arg F_{n}(x)$. Moreover, this branch may be chosen nondecreasing on $\mathbb R$ satisfying
	\[
      \theta_{n}(x+1)= \theta_{n}(x)+ 2 \pi \nu.
	\]
\end{theorem}

\begin{proof} Let $m=\lfloor \lambda \rfloor +1.$ We first prove the claim for $F_{m}$ and then inductively for $F_{n}$, $n\geq m$. 
	
	\emph{Base case.}
    Since $\lambda>1$, we have 
	$
	1<\frac{m}{\lambda}< 2.
	$
Therefore 
	\[
	B_m\bigl(\lambda(x-k)\bigr)\neq0 \quad\Leftrightarrow\quad
	|x-k|<\frac{m}{2\lambda}<1.
	\]
	Consequently, when $0 \le x<1$, the only potentially nonzero terms in \eqref{def:F_n} are those with $k=0$ and $k=1$. Therefore
	\begin{align*}
	F_{m}(x)
&=
B_m(\lambda x)
+
B_m\bigl(\lambda(x-1)\bigr)\myexp{2\pi i \nu}\\
	&=
	B_m(\lambda x)
+
B_m\bigl(\lambda(1-x)\bigr)\myexp{2\pi i \nu}.
	\end{align*}
	Set
	\[
	a(x)=B_m\left(\lambda x\right),
	\qquad
	c(x)=B_m\left(\lambda(1-x)\right), \quad\text{ and, }\quad \alpha=2 \pi \nu.
	\]
	Since $m/\lambda>1$, the two supports overlap enough that $a(x)$ and
	$c(x)$ are not both zero on $\itvco{0,1}$. 
%	More explicitly, if both were
%	zero, then we would have
%	\[
%	x\ge \frac{mb}{2}
%	\qquad\text{and}\qquad
%	1-x\ge \frac{mb}{2},
%	\]
%	which would imply
%	\[
%	1\ge mb,
%	\]
%	contradicting $mb>1$. 
Thus
$$
F_m(x)=a(x)+c(x)\myexp{i\alpha}.
$$
is a nontrivial nonnegative linear combination of $1$ and
$\myexp{i\alpha}$. Since $0<\alpha<\pi$, these two vectors span a sector
of angle strictly smaller than $\pi$. Consequently,
$F_m(x)\neq0,
$ for $0 \le x<1.
$ Define $\theta_m(x)$ as the unique argument of $F_m(x)$ in $[0, \alpha]$.
% If $ m\ge3$, the following derivatives exist everywhere. If \(m=2\), they exist except at finitely many points. 
Then 
$$
a'(x)\le0,
\qquad
c'(x)\ge0
$$
for all $x$ in the domain of differentiability.
Hence
$$
\begin{aligned}
	\operatorname{Im}\!
	\left(
	F_{m}'(x)\overline{F_{m}(x)}
	\right)
	&=
	\operatorname{Im}
	\left[
	(a'(x)+c'(x)\myexp{i\alpha})
	(a(x)+c(x)\myexp{-i\alpha})
	\right]
	\\
	&=
	\bigl(a(x)c'(x)-a'(x)c(x)\bigr)\sin\alpha.
\end{aligned}
$$
Since
% $$
% a(x),c(x)\ge0,
% \qquad
% a'(x)\le0,
% \qquad
% c'(x)\ge0,
% $$
%and
$\sin\alpha>0$, it follows that
$$
\operatorname{Im}\!
\left(
F_{m}'(x)\overline{F_{m}(x)}
\right)\ge0.
$$
% Thus, for a continuous branch
% $$
% \theta_m(x)=\arg F_m(x),
% $$
Thus, 
$$
\begin{aligned}
    \theta_m'(x)
	=
	\operatorname{Im}
	\left(
	\frac{F_{m}'(x)}{F_{m}(x)}
	\right)
	=
	\frac{
		\operatorname{Im}\!
		\left(
		F_{m}'(x)\overline{F_{m}(x)}
		\right)
	}{
		|F_{m}(x)|^2
	}
	\ge0,
\end{aligned}
$$
on every interval of differentiability, and continuity of \(\theta_m\) across the finitely many breakpoints implies that \(\theta_m\) is nondecreasing on \([0,1]\).
At the endpoints,
$$
F_{m}(0)=B_m(0)>0,
\qquad
F_{m}(1)=B_m(0)\myexp{i\alpha}.
$$
Hence we may choose the continuous branch on $[0,1]$ so that
$$
\theta_m(0)=0,
\qquad
\theta_m(1)=\alpha.
$$
The quasi-periodicity
$
F_{m}(x+1)=\myexp{i\alpha}F_{m}(x),
$
then permits the extension
$$
\theta_m(x+1)
=
\theta_m(x)+\alpha.
$$
This yields a continuous nondecreasing branch on all of $\mathbb R$.

\emph{Induction step.}
Assume that, for some $n\ge m$, $F_{n}(\cdot)$ is nonzero on
$\mathbb R$, and that there exists a continuous nondecreasing branch
$$
\theta_n(x)=\arg F_n(x)
$$
satisfying
$$
\theta_n(x+1)
=
\theta_n(x)+\alpha.
$$
We prove the same assertions for $F_{n+1}(\cdot)$.
Using
$
B_{n+1}=B_n*B_1,
$
we obtain
\begin{align}\label{eq:F_n+1}
	F_{n+1}(x)
	&=
	\sum_{k\in\mathbb Z}
	B_{n+1}\bigl(\lambda(x-k)\bigr)\myexp{i\alpha k}
	\nonumber \\
	&=
	\lambda
	\int_{-1/(2\lambda)}^{1/(2\lambda)}
	\sum_{k\in\mathbb Z}
	B_n\bigl(\lambda(x-u-k)\bigr)\myexp{i\alpha k}\,du
	\nonumber \\
	&=
	\lambda
	\int_{-1/(2\lambda)}^{1/(2\lambda)}
	F_n(x-u)\,du
	\nonumber \\
	&=
	\lambda
	\int_{x-1/(2\lambda)}^{x+1/(2\lambda)}
	F_n(y)\,dy.
\end{align}

Set
$$
p=x-\frac1{2\lambda},
\qquad
q=x+\frac1{2\lambda}.
$$
By the induction hypothesis, $\theta_n(\cdot)$ is
nondecreasing. Thus, for $p\le y\le q$,
$$
\theta_n(p)
\le
\theta_n(y)
\le
\theta_n(q).
$$
Since $q<p+1$,
$$
\theta_n(q)
\le
\theta_n(p+1)
=
\theta_n(p)+\alpha.
$$
Therefore
$$
0\le
\theta_n(q)-\theta_n(p)
\le\alpha<\pi.
$$
Hence the vectors \(F_n(y)\), \(y\in[p,q]\), lie in a sector of width strictly smaller than \(\pi\).
Setting
$$
\phi(x)
=
\frac{
	\theta_{n}(p)
	+
	\theta_{n}(q)
}{2},
$$
we obtain that for every \(y\in[p,q]\),
$
\left|
\theta_{n}(y)-\phi(x)
\right|
<\frac{\pi}{2},
$
and hence
$$
\operatorname{Re}
\left(
\myexp{-i\phi(x)}F_n(y)
\right)>0.
$$
Therefore
$$
\operatorname{Re}
\left(
\myexp{-i\phi(x)}F_{n+1}(x)
\right)
=
\lambda
\int_p^q
\operatorname{Re}
\left(
\myexp{-i\phi(x)}F_{n}(y)
\right)\,dy
>0.
$$
Thus \(F_{n+1}(x)\neq0\), and its argument lies in
$$
\bigl[
\theta_{n}(p),
\theta_{n}(q)
\bigr].
$$
We therefore define $\theta_{n+1}(x)$ as the unique real argument of $F_{n+1}(x)$ in this interval. Since $F_{n+1}(x)$ and the endpoints of the interval depend continuously on $x$, this defines a continuous branch satisfying
$$
\theta_{n}(p)
\le
\theta_{n+1}(x)
\le
\theta_{n}(q).
$$
%%%
% \operatorname{Re}
% \left(
% \myexp{-i\phi}F_n(y)
% \right)>0, \quad \text{ for all } p \leq y \leq q.
% $$
% Consequently,
% $$
% \operatorname{Re}
% \left(
% \myexp{-i\phi}
% \int_p^qF_n(y)\,dy
% \right)>0.
% $$
% and therefore
% $$
% F_{n+1}(x)\neq0.
% $$
% Moreover, since
% $$
% \{F_n(y):y\in\itvcc{p,q}\}
% $$
% is contained in the sector bounded by the angles
% $$
% \theta_{n}(p)
% \quad\text{and}\quad
% \theta_{n}(q),
% $$
% the nonzero integral in \eqref{eq:F_n+1} lies in the same sector. Since this sector has width strictly smaller than \(\pi\), there is a unique real argument of \(F_{n+1, \nu}(x)\) in
% $$
% [\theta_{n}(p),\theta_{n}(q)].
% $$
% We define \(\theta_{n+1}(x)\) to be this argument. The uniqueness and continuous dependence of the sector and of \(F_{n+1, \nu}\) on \(x\) imply that \(\theta_{n+1, \nu}\) is continuous.
% We may therefore choose
% a continuous branch $\theta_{n+1}$ satisfying
% $$
% \theta_{n, \nu}(p)
% \le
% \theta_{n+1, \nu}(x)
% \le
% \theta_{n, \nu}(q).
% $$
When $x$ is replaced by $x+1$,  the corresponding sector is shifted by exactly
$\alpha$. Together with the quasi-periodicity of $F_{n+1}$, we get
$$
\theta_{n+1}(x+1)
=
\theta_{n+1}(x)+\alpha.
$$
It remains to prove that this branch is nondecreasing. From \eqref{eq:F_n+1} it follows that 
$$
F_{n+1}'(x)
=
\lambda
\left(
F_{n}(q)-F_{n}(p)
\right).
$$
Hence
$$
\begin{aligned}
	\operatorname{Im}
	\left(
	F_{n+1}'(x)
	\overline{F_{n+1}(x)}
	\right)
	=
	\lambda^2
	\int_p^q
	\operatorname{Im}
	\left[
	\left(
	F_{n}(q)-F_{n}(p)
	\right)
	\overline{F_n(y)}
	\right]\,dy.
\end{aligned}
$$
Writing 
$
F_n(y)
=
|F_n(y)|\myexp{i\theta_n(y)},
$
we get 
$$
\begin{aligned}
	\operatorname{Im}
	\left(
F_n(q)\overline{F_n(y)}
	\right)
	=
|F_n(q)F_n(y)|
	\sin\!
	\left(
	\theta_{n}(q)-\theta_{n}(y)
	\right)
	\ge0,
\end{aligned}
$$
because
$
0\le
\theta_{n}(q)-\theta_{n}(y)
\le\alpha<\pi.
$
Similarly,
$$
\begin{aligned}
	-\operatorname{Im}
	\left(
F_n(p)\overline{F_n(y)}
	\right)
=|F_n(p)F_n(y)|
	\sin\!
	\left(
	\theta_{n}(y)-\theta_{n}(p)
	\right)
	\ge0.
\end{aligned}
$$
It follows that
$$
\operatorname{Im}
\left[
\left(
F_{n}(q)-F_{n}(p)
\right)
\overline{F_n(y)}
\right]\ge0, \quad y\in[p,q]. $$
Consequently,
$$
\operatorname{Im}
\left(
F_{n+1}'(x)
\overline{F_{n+1}(x)}
\right)\ge0.
$$
Since $F_{n+1}(x)\neq0$,

$$
\begin{aligned}
    \frac{d}{dx}\theta_{n+1}(x)
	=
	\operatorname{Im}
	\left(
	\frac{F_{n+1}'(x)}
	{F_{n+1}(x)}
	\right)
	=
	\frac{
		\operatorname{Im}
		\left(
		F_{n+1}'(x)
		\overline{F_{n+1}(x)}
		\right)
	}{
		|F_{n+1}(x)|^2
	}
	\ge0.
\end{aligned}
$$
 Hence
$\theta_{n+1}(\cdot)$ is nondecreasing.
This completes the induction step.
\end{proof}

By symmetry of the Zak transform, the \Cref{thm:zeros-zak-bspline-medium-lambda} implies that for $1<\lambda<n$, the Zak transform $Z_\lambda B_n$ has no zeros in the domain $\itvco{0,1}\times \itvoo{1/2,1}$ either. However, we still need to consider the boundary of the reduced fundamental domain $\itvcc{0,1/2}^2$ as this is not covered by \Cref{thm:zeros-zak-bspline-medium-lambda}.
This is the subject of the next section.

\section{The Zak transform $Z_{\lambda} B_n$ along lines}
\label{sec:zeros-boundary-lines}

In \cite{MR3572909} it was shown that, for $\lambda < 2/3$, the Zak transform $Z_\lambda B_n$ vanishes on $\round{1/\lambda}-1$ horizontal lines in the fundamental domain $\itvco{0,1}^2$ whenever $1/\lambda$ is sufficiently close to an integer larger than or equal to two. Our first result gives a characterization of the zero set of $Z_\lambda B_n$ along these lines.

\begin{proposition}
    \label{thm:zeros-zak-bspline}
Let $n \ge 2$ be an integer and $0<\lambda < 1$ a real number. Set $R=\round{1/\lambda}$ and $F=\sfrac{1/\lambda}$. Fix  $\nu = \frac{s}{R}, s=1,\ldots,R-1$.
    Then the following holds.
    \begin{enumerate}[label=(\roman*)]
        \item If $F=0$, then
        \[
            Z_{\lambda}B_n(x,\nu)=0,
            \qquad x\in\R.
        \]
		\label{itm:zeros-zak-bspline-F=0}
        \item Suppose $0<\abs{F}\le\frac1n$.
        If $nR$ is even, then, for $x\in\itvco{0,1}$,
        \[
            Z_{\lambda}B_n(x,\nu)=0
            \qquad\Longleftrightarrow\qquad
            x\in
            \itvcc*{
                \frac{n\abs F}{2},
                1-\frac{n\abs F}{2}
            }.
        \]
        If $nR$ is odd, then, for $x\in\itvco{0,1}$,
        \[
            Z_{\lambda}B_n(x,\nu)=0
            \quad\Longleftrightarrow\quad
            x\in
            \itvcc*{
                0,\frac{1-n\abs F}{2}
            }
            \cup
            \itvco*{
                \frac{1+n\abs F}{2},1
            }.
        \]
		\label{itm:zeros-zak-bspline-0<|F|<=1/n}
        \item Suppose $\frac1n<\abs F\le\frac12$.
If $\nu \in \set{1/R,2/R, \dots, (R-1)/R} \setminus \{\frac12\}$, then
\[
    Z_{\lambda}B_n(x,\nu)\neq 0,
    \qquad x\in\itvco{0,1}.
\]
If $\nu=\frac12$, then, for $x\in\itvco{0,1}$,
\[
    Z_{\lambda}B_n(x,1/2)=0
    \quad\Longleftrightarrow\quad
    x=\frac12.
\]
		\label{itm:zeros-zak-bspline-1/n<F<=1/2}
    \end{enumerate}
\end{proposition}

\begin{proof}
    Since $R\nu=s\in\Z$, we may apply \Cref{lem:aliasing} with $M=R$. Set $d=\abs F$.

	If $F=0$, then the last assertion of \Cref{lem:aliasing} gives
    \[
        Z_{\lambda}B_n(\cdot,\nu) = 0,
    \]
    which proves \ref{itm:zeros-zak-bspline-F=0}.

    Suppose now that $F\neq0$. Assume first that
    $ 0<d\le\frac1n$.
    Then
    $\frac1d\ge n$.
    Hence, by \Cref{thm:zeros-zak-bspline-large-lambda},
    \[
        Z_{1/d}B_n(x,\nu)=0
        \quad\Longleftrightarrow\quad
        x\in
        \itvcc*{
            \frac{nd}{2},
            1-\frac{nd}{2}
        }.
    \]
    If $nR$ is even, the first identity in \eqref{eq:aliasing-lemma} gives the first
    assertion in \ref{itm:zeros-zak-bspline-0<|F|<=1/n}. If $nR$ is odd, the second identity in \eqref{eq:aliasing-lemma}
    shifts this interval by $1/2$. Restricting to
    $x\in\itvco{0,1}$ gives
    \[
        \itvcc*{
            0,\frac{1-nd}{2}
        }
        \cup
        \itvco*{
            \frac{1+nd}{2},1
        }.
    \]
    This proves \ref{itm:zeros-zak-bspline-0<|F|<=1/n}.

    Finally, suppose that $\frac1n<d\le\frac12$.
    Then $2\le\frac1d<n$.
	For $0<\nu<1/2$, the non-vanishing result for
    $1<1/d<n$ in \Cref{thm:zeros-zak-bspline-medium-lambda} with $\lambda=1/d$ gives
    \[
        Z_{1/d}B_n(x,\nu)\neq0,
        \qquad x\in\R.
    \]
    By symmetry of the Zak transform,
    the same holds for $1/2<\nu<1$.

    It remains only to consider $\nu=1/2$. In this case $R$ is even,
    hence $nR$ is even and there is no translation in \eqref{eq:aliasing-lemma}. Moreover,
    since $1<1/d<n$,
    \[
        Z_{1/d}B_n(x,1/2)=0
        \quad\Longleftrightarrow\quad
        x=\frac12.
    \]
    This proves \ref{itm:zeros-zak-bspline-1/n<F<=1/2}.
\end{proof}

We now turn the boundary of the reduced fundamental domain $\itvcc{0,1/2}^2$. Hence, we consider the Zak transform restricted to the lines $x=0$, $x=1/2$, $\nu=0$, and $\nu=1/2$ in the $(x,\nu)$-plane. For $\lambda \ge n $, we already have a full characterization of the zeros of \(Z_\lambda B_n\) in \Cref{thm:zeros-zak-bspline-large-lambda}, hence we restrict our attention to \(0<\lambda<n\). 

\begin{proposition}
    \label{lem:non-zero-Zak_along-boundary}
Let \(n\ge2\). Then
\begin{enumerate}[label=(\roman*)]
\item If $0<\lambda<n$, then
$$
Z_\lambda B_n(x,0)\neq0
\qquad\text{for all } x\in\mathbb R,
$$
\label{itm:non-zero-Zak_along-boundary_nu=0}
\item If $\lambda > 1$, then
$$
Z_\lambda B_n(0,\nu)\neq0
\qquad\text{for all } \nu\in\mathbb R.
$$
\label{itm:non-zero-Zak_along-boundary_x=0}
\end{enumerate}
\end{proposition}

\begin{proof} For \ref{itm:non-zero-Zak_along-boundary_nu=0}, let $0<\lambda<n$. For every $x\in\R$ there exists
    $k\in\Z$ such that $\abs{x-k}\le1/2$. Since
    \[
        \frac{n}{2\lambda}>\frac12,
    \]
    we have
    \[
        \abs{\lambda(x-k)}<\frac n2,
    \]
    and hence $B_n(\lambda(x-k))>0$. Since all terms in
    $Z_\lambda B_n(x,0)$ are non-negative, it follows that
    \[
        Z_\lambda B_n(x,0)>0.
    \]

    For \ref{itm:non-zero-Zak_along-boundary_x=0}, suppose first that $1<\lambda<n$. \Cref{thm:zeros-zak-bspline-medium-lambda} gives
    \[
        Z_\lambda B_n(0,\nu)\neq0,
        \qquad 0<\nu<1/2.
    \]
    The same holds for $1/2<\nu<1$ by complex conjugation. The cases
    $\nu=0$ and $\nu=1/2$ follow from \ref{itm:non-zero-Zak_along-boundary_nu=0} and
    \Cref{lem:non-zero-Zak_along-boundary_onehalf_x}, respectively.
    By periodicity in $\nu$, this proves the assertion for all $\nu\in\R$.

    If $\lambda\ge n$, only the term $k=0$ contributes at $x=0$, and
    therefore
    \[
        Z_\lambda B_n(0,\nu)
        =
        \sqrt{\lambda}\,B_n(0)>0, \qquad\text{for all } \nu\in\mathbb R.
    \]
\end{proof}

% \begin{proof}
% 	If \(\lambda\ge n/2\), then \(B_n(\lambda k)=0\) for \(k\neq0\), and hence
% $
% Z_\lambda B_n(0,\nu)=\sqrt{\lambda}\,B_n(0)>0.
% $
% Assume \(1<\lambda<n/2\). By the Poisson summation formula,
% $$
% Z_\lambda B_n(0,\nu)
% =\frac1{\sqrt{\lambda}}
% \sum_{m\in\mathbb Z}
% \sinc^n\!\left(\frac{m+\nu}{\lambda}\right),
% \qquad
% \sinc(t)=\frac{\sin(\pi t)}{\pi t}.
% $$
% If \(n\) is even, every term is nonnegative, while the \(m=0\) term is positive; hence \(Z_\lambda B_n(0,\nu)>0\).
% Let \(n\) be odd. Since terms with \(|m+\nu|\le\lambda\) are nonnegative, for $\nu \in [0, 1/2]$, we have 
% $$
% \begin{aligned}
% 	\sqrt{\lambda}\,Z_\lambda B_n(0,\nu)
% 	&\ge \left(\frac2\pi\right)^n
% 	-\frac{\lambda^n}{\pi^n}
% 	\sum_{|m+\nu|>\lambda}\frac1{|m+\nu|^n}\\
% 	&\ge
% 	\frac1{\pi^n}
% 	\left[
% 	2^n-2\left(1+\frac{\lambda}{n-1}\right)
% 	\right]>0,
% \end{aligned}
% $$
% where the last inequality follows from \(\lambda<n/2\). Thus  \(Z_\lambda B_n(0,\nu)\neq0\).
% \end{proof}

\begin{remark}
For \(0<\lambda<1\), the function \(Z_\lambda B_n(0,\nu)\) may vanish. In particular, if
$
\lambda=\frac1q,$ for $ q\ge2,
$
then 
$$
Z_{1/q}B_n\left(0,\frac rq\right)=
\sqrt q\sum_{m\in\mathbb Z}\sinc^n(qm+r)=0, \qquad r=1,\ldots,q-1.
$$
For even \(n\), these are precisely the zeros of \(Z_\lambda B_n(0,\nu)\) in the regime \(0<\lambda<1\); for odd \(n\), additional zeros may occur due to cancellation.
\end{remark}

\begin{proposition}
    \label{lem:non-zero-Zak_along-boundary_onehalf_x}
    Let $\lambda>0$ and $n\ge2$, and set $R = \round{1/\lambda}$ and $F = \abs{\sfrac{1/\lambda}}$. 
    \begin{enumerate}[label=(\roman*)]
        \item If $R$ is even, then, for $x\in\itvco{0,1}$,
        \[
            Z_\lambda B_n(x,1/2)=0
            \quad \Longleftrightarrow \quad
            x\in
            \set{1/2}
            \cup
            \itvcc*{\frac n2F,1-\frac n2F};
        \]
        \item If $R$ is odd, then, for $x\in\itvco{0,1}$,
        \[
            Z_\lambda B_n(x,1/2)=0
            \quad \Longleftrightarrow \quad
            x\in
            \set{1/2}
            \cup
            \itvcc*{
                \frac n2(1-F),
                1-\frac n2(1-F)
            }.
        \]
    \end{enumerate}
    where the interval is understood as empty if its left endpoint exceeds its right endpoint.
\end{proposition}

\begin{proof}
	The two assertions can be combined into a single statement. Set
    \[
        d=\operatorname{dist}\left(\frac1\lambda,2\Z\right).
    \]
    Then we have to show that, for $x\in\itvco{0,1}$,
    \begin{equation}		
		\label{eq:zeros-zak-bspline-onehalf-combined}
        Z_\lambda B_n(x,1/2)=0
        \quad \Longleftrightarrow \quad
        x\in
        \set{1/2}
        \cup
        \itvcc*{\frac{nd}{2},1-\frac{nd}{2}},
    \end{equation}
   Let $b=\frac1\lambda$,
    and choose $M\in2\Z$ such that
    $d=\abs{b-M}$. Since $M/2 \in\Z$,
    we may apply \Cref{lem:aliasing} with $\nu=1/2$. Moreover, $M$
    is even, and hence $nM$ is even. Thus, if $d>0$,
    \begin{equation}
		\label{eq:aliasing-lemma-1/2-even}
		Z_\lambda B_n(x,1/2)
		=
		C\,Z_{1/d}B_n(x,1/2),
		\qquad C\neq0.
	\end{equation}
    If $d=0$, \Cref{lem:aliasing} gives
    \[
        Z_\lambda B_n(\cdot,1/2)\equiv0,
    \]
    which agrees with \eqref{eq:zeros-zak-bspline-onehalf-combined}.

    Suppose first that $0<d\le\frac1n$.
    Then $1/d\ge n$, and hence
    \Cref{thm:zeros-zak-bspline-large-lambda} gives
    \[
        Z_{1/d}B_n(x,1/2)=0
        \qquad\Longleftrightarrow\qquad
        x\in
        \itvcc*{\frac{nd}{2},1-\frac{nd}{2}}.
    \]
    Together with \eqref{eq:aliasing-lemma-1/2-even}, this proves the assertion \eqref{eq:zeros-zak-bspline-onehalf-combined} in this case.

    It remains to consider $d>\frac1n$.
    Then $1\le\frac1d<n$.
    We claim that, for every $1\le\mu<n$,
	\begin{equation}
		\label{eq:Z_mu_B_n_1/2}
        Z_\mu B_n(x,1/2)
        \begin{cases}
            >0,&0<x<1/2,\\
            =0,&x=1/2,\\
            <0,&1/2<x<1.
        \end{cases}
	\end{equation}	
    To see this, let
    \[
        m=\floor{\mu}+1.
    \]
    Then $m\le n$ and
    \[
        1<\frac m\mu\le2.
    \]
    Hence, for $0<x<1$, only the terms $k=0$ and $k=1$ can occur in
    $Z_\mu B_m(x,1/2)$, and therefore
    \[
        Z_\mu B_m(x,1/2)
        =
        \sqrt{\mu}
        \left(
            B_m(\mu x)
            -
            B_m(\mu(1-x))
        \right).
    \]
    Since $B_m$ is even and strictly decreasing on
    $\itvcc{0,m/2}$, this expression is positive for
    $0<x<1/2$, vanishes at $x=1/2$, and is negative for
    $1/2<x<1$.

    The assertion \eqref{eq:Z_mu_B_n_1/2} is preserved when the order is increased. Indeed, writing
    \[
        G_n(x)=Z_\mu B_n(x,1/2),
    \]
    we have
    \begin{equation}
		\label{eq:Z_mu_B_n_1/2_induction}
		G_{n+1}(x)
        =
        \mu
        \int_{x-\frac1{2\mu}}^{x+\frac1{2\mu}}
        G_n(y)\,dy.
	\end{equation}

    Moreover,
	\[
		G_n(-x)=G_n(x), \qquad G_n(1-x)=-G_n(x).
	\]
	Thus, by evenness and the induction hypothesis,
	\[
		G_n(x)>0, \qquad -\frac12<x<\frac12.
	\]
	If $0<x<1/2$, set
	\[
		p=x-\frac1{2\mu}, \qquad q=x+\frac1{2\mu}.
	\]
	Since $\mu\ge1$, we have
	\[
		-\frac12<p<q<1.
	\] 
	If $q\le1/2$, then $G_n>0$ on the interior of
    $\itvcc{p,q}$, and hence the integral in \eqref{eq:Z_mu_B_n_1/2_induction} is positive.
    If $q>1/2$, the antisymmetry gives
    \[
        \int_{1-q}^{1/2}G_n(y)\,dy
        +
        \int_{1/2}^{q}G_n(y)\,dy
        =0,
    \]
    and therefore
    \[
        \int_p^qG_n(y)\,dy
        =
        \int_p^{1-q}G_n(y)\,dy>0,
    \]
    since
    \[
        p<1-q
        \qquad\Longleftrightarrow\qquad
        x<\frac12.
    \]
    This proves \eqref{eq:Z_mu_B_n_1/2} by induction.

    Applying \eqref{eq:Z_mu_B_n_1/2} with $\mu=1/d$ and using \eqref{eq:aliasing-lemma-1/2-even}, we conclude that
    $x=1/2$ is the unique zero when $d>1/n$. Since in this case
    \[
        \itvcc*{\frac{nd}{2},1-\frac{nd}{2}}
    \]
    is empty, the asserted formula \eqref{eq:zeros-zak-bspline-onehalf-combined} follows.
    % Finally, if
    % \[
    %     R=\round{b},
    %     \qquad
    %     F=\abs{\sfrac b},
    % \]
    % then the distance from $b$ to the nearest even integer is
    % \[
    %     d=
    %     \begin{cases}
    %         F,&R\text{ even},\\
    %         1-F,&R\text{ odd}.
    %     \end{cases}
    % \]
    % Substituting these expressions for $d$ gives
    % \ref{item:R-even} and \ref{item:R-odd}.
\end{proof}

We postpone the investigation of zeros of the Zak transform along the line $x = 1/2$ to~\Cref{sec:zeros-set-Bn} and \Cref{sec:palindromic-polynomials} as this requires a more detailed analysis. %In fact, for higher order splines, i.e., $n > 2$, we will only give a complete characterization of the zeros of $Z_\lambda B_n$ along the line $x = 1/2$ for $\lambda \ge 1$.

\section{The zero set of the Zak transform $Z_\lambda B_n$}
\label{sec:zeros-set-Bn}

For $\lambda\ge1$, the preceding results on the interior and the boundary of the reduced fundamental domain combine to give a complete description
of the zero set: 

\begin{corollary}
    \label{cor:zero-set-Zak-B_n-for-lambda-ge-1}

    Let $n\ge2$ and $\lambda\ge1$. Then, for
    $(x,\nu)\in\itvco{0,1}^2$,
    \[
        Z_\lambda B_n(x,\nu)=0
        \quad\Leftrightarrow\quad
        \begin{cases}
            (x,\nu)=(1/2,1/2),
            & 1\le\lambda<n,\\[2mm]
            x\in\itvcc{\frac{n}{2\lambda},\,1-\frac{n}{2\lambda}}
            \text{ and }\nu\in\itvco{0,1},
            & \lambda\ge n.
        \end{cases}
    \]
\end{corollary}
\begin{proof}
	The case $\lambda=1$ is the classical case. For $1<\lambda<n$, the assertion follows from
    \Cref{thm:zeros-zak-bspline-medium-lambda,lem:non-zero-Zak_along-boundary,lem:non-zero-Zak_along-boundary_onehalf_x} and the symmetries of the
    Zak transform. For $\lambda\ge n$, the assertion
    follows from \Cref{thm:zeros-zak-bspline-large-lambda}.
\end{proof}

For $\lambda<1$, the zero set is considerably richer. We restrict our
attention to the hat spline $B_2$, for which the zeros can still be
described explicitly. Besides the rational frequencies $\nu=\frac{s}{R}$ from \Cref{thm:zeros-zak-bspline},
which give intervals of zeros in the $x$-variable, there is a second
family of isolated zeros on the line $x=1/2$, determined by a
cotangent equation. The case $\nu=0$ has already been treated in \Cref{lem:non-zero-Zak_along-boundary}, where it was shown that
$$
Z_\lambda B_2(x,0)\neq0
\qquad\text{for all }x\in\R.
$$
Hence, in the following theorem we restrict to $\nu\in(0,1)$.

\begin{theorem}
	\label{thm:char-zak-B2-lambda-le-1}
    Let $0<\lambda < 1$, $R = \round{1/\lambda}$, and $F = {\sfrac{1/\lambda}}$.    
%    Assume that $b=R+F$, where $R>1$ and $0< F <\frac12$. 
   For $\nu\in(0,1)$, the following assertions are true:
    \begin{enumerate}[label=(\roman*)]
        \item For $x\in\itvco{0,|F|}\cup \itvoo{1-|F|, 1}$, the Zak transform  $Z_\lambda B_2(x , \nu ) $ has no zero. 
        \item For $x\in \itvcc{|F|,1-|F|}$,  the zeros are characterized as follows
        \begin{enumerate}[label=\alph*)]
            \item  If $\nu\in\left\{\frac{s}{R}:s=1,\dots, R-1\right\},$ then 
             $Z_\lambda B_2(x , \nu ) = 0$  for every $x\in \itvcc{|F|,1-|F|}$. 
            \item If $\nu\notin\left\{\frac{s}{R}:s=1,\dots, R-1\right\},$ then   $Z_\lambda B_2(x , \nu ) $ is zero if and only if 
            $$x = \frac12, \quad\text{ and } \cot(\pi \nu) + 2F \cot(\pi R \nu)=0.$$
         %     In particular if $R$ is odd, then 
         % $$Z_{\lambda}B_2(x,\frac12)=0 \Longleftrightarrow x=\frac12.$$
         \end{enumerate}
    \end{enumerate}
\end{theorem}

\begin{proof} For notational convenience, let \(b=1/\lambda\). We consider only the case \(F>0\), leaving the case \(F\leq 0\) to the reader. By the symmetry of the Zak transform, it suffices to consider $0 \le x \le 1/2$. 
%\begin{equation*}
%	Z_\lambda B_2(1- x , \nu ) =
%    &=\sum_{k\in \mathbb{Z}} B_2(\frac{1- x-k}{b}) \myexp{2\pi i k \nu}\\
%	&= \myexp{2\pi i \nu}\sum_{k\in \mathbb{Z}} B_2(\frac{- x-k}{b}) \myexp{2\pi i k \nu}\\
%	&= \myexp{2\pi i  \nu}\sum_{k\in \mathbb{Z}} B_2(\frac{ x+k}{b}) \myexp{2\pi i k \nu}\\
%	&= 
%    \myexp{2\pi i  \nu} \overline{Z_\lambda B_2(x ,\nu )}.
%\end{equation*}
 Since the hat spline is zero outside $\itvoo{-1,1}$, we have
\begin{align}
	\label{eq:z-ell}
	\sqrt{b}Z_\frac1{b}B_2(x , \nu )
%	&= 
 %    \sum_{k\in \mathbb{Z}} B_2\!\left(\frac{x-k}{b}\right) \myexp{2\pi i k \nu}
	% = \sum_{k=\lceil{ -b + x \rceil}}^{\lfloor b+ x \rfloor} B_2\!\left(\frac{x-k}{b}\right) \myexp{2\pi i k \nu} \nonumber \\
	&= \sum_{k=\lceil{-b + x  \rceil }}^{\lfloor  x \rfloor}
	\left(1- \frac{x-k}{b}\right) \myexp{2\pi i k \nu}
	+ \sum_{k=\lfloor{   x \rfloor}+1}^{\lfloor  b+ x \rfloor}
	\left(1+ \frac{x-k}{b}\right)\myexp{2\pi i k \nu}.
	\nonumber
\end{align}
%&=&\sum_{k=\lceil{ -b+ (x- ) \rceil}}^{0} 1-(\frac1b(x- -k)) \myexp{2\pi i k \nu}
%+ \sum_{k=1}^{\lfloor b + (x-\frac{\ell}{q})\rfloor} 1+(\frac1b(x-\frac{\ell}{q}-k)) \myexp{2\pi i k \nu} 
We partition the interval $\itvcc{0,\frac12}$ into the two subintervals $\itvcc{0, F}$ and $\itvoc{F, \frac12}$. 
This partition differs slightly at the boundary points from the one used in the proposition, but it is the natural choice for the proof.

\emph{Case 1:}
For $x \in \itvcc{0, F}$, we have 
\[
	 \lceil -b + x \rceil = -\lfloor b - x\rfloor =-R- \lfloor F-x \rfloor = -R, \qquad   \lfloor b +  x \rfloor = R.
\]
%and if $F(b)<0$, 
%$$ \lfloor b + (x-\frac{\ell}{q})\rfloor = \lfloor b \rfloor = R(b), \qquad  \lceil -b+(x-\frac{\ell}{q}) \rceil = - \lfloor b -  (x-\frac{\ell}{q})\rfloor = - \lfloor b \rfloor= - R(b)$$
Therefore, 
\begin{equation}
	\label{2403b}
    \sqrt{b} Z_\frac1{b}B_2(x , \nu )
	=
	\sum_{k=-R}^{0} \left(1-\frac{x-k}{b}\right) \myexp{2\pi i k \nu}
	+
	\sum_{k=1}^{R} \left(1+\frac{x-k}{b}\right) \myexp{2\pi i k \nu}.
\end{equation}
	Considering $z=\myexp{2 \pi i \nu}$, we have 
\begin{equation}
	\label{2304a}
\sqrt{b} Z_\frac1{b}B_2(x , \nu )  = 
\frac{1}{b}
\left[
b-x 
+
2\sum_{k=1}^{R}(b-k)\cos(2\pi k\nu)
+
2ix\sum_{k=1}^{R}\sin(2\pi k\nu)
\right].
\end{equation}
Assume that for some $x\in \itvcc{0, F}$ the above function is zero. Then its imaginary part must also be zero. Hence either $x=0$ or $\sum_{k=1}^{R} \sin(2 \pi k \nu)=0$. 

For $x=0$, the equation \eqref{2304a} has the form
\begin{align}
	\label{eq:real-part-4.3}
	b \sqrt{b} Z_\frac1{b}B_2(0 , \nu )&= 
	b + 2\sum_{k=1}^{R} (b-k) \cos(2\pi k \nu) \nonumber \\
	&= \left( \frac{\sin(\pi R\nu)}{\sin(\pi \nu)}\right)^2 + F\, \frac{\sin( (2R+1)\pi\nu)}{\sin(\pi \nu)}.
\end{align}
Using trigonometric identities, we obtain
\[
	b \sqrt{b} Z_\frac1{b}B_2(0 , \nu )= \frac{(1 - F)\sin^2(\pi R \nu) + F \sin^2(\pi (R+1)\nu)}{\sin^2(\pi \nu)},
\]
which is nonzero for $v\in\itvoo{0,1/2}$. Now 
assume that $\sum_{k=1}^{R} \sin(2 \pi k \nu)=0$. We use the identity
\begin{equation}
	\label{eq:identity-for-sin}
	\sum_{k=1}^{R } \sin(2\pi  k \nu)=
	\begin{cases}
		0, & \nu=0, \\
		\csc(\pi \nu)\sin(\pi R\nu)\sin(\pi (R + 1)\nu), & \nu\neq 0.
	\end{cases}
\end{equation}
Therefore, for $\nu\in \itvoo{0,1}$, the sum $\sum_{k=1}^{R} \sin(2 \pi k \nu)$ is zero
if and only if
$$\nu \in\left\{ \frac{s}{R}: s=1,\dots, R -1 \right\}\cup \left\{ \frac{s}{R+1}: s=1,\dots, R \right\}.$$ 
First assume that  $\nu = \frac{s}{R}$ for some $s=1,\dots, R -1$.  Then 
	\begin{eqnarray}\label{eq:legendre-equations}
		\sum_{k=1}^R \cos(2 \pi k \nu)=0, \quad \sum_{k=1}^{R } k \cos(2\pi  k \nu)
		= R /2.
	\end{eqnarray}
	It follows that
	\begin{equation*}
		\Re(b \sqrt{b} Z_\frac1{b}B_2(x,\nu) )= 
		b-x -R= F-x.
	\end{equation*}
	Hence the real part vanishes only if $x=F$. 
 Now assume that $\nu = \frac{s}{R +1 }$ for some $s=1,\dots, R $, then 
	$$  \sum_{k=1}^{R } \cos(2\pi  k \nu) = -1, $$
	and
	$$\sum_{k=1}^{R } k \cos(2\pi  k \nu)
	=   \frac {(R +1 )\cos(2 \pi R \nu ) - R \cos(2 \pi (R + 1)\nu) - 1}{4\sin^2(\pi \nu)}= -\frac{R+1}{2}.$$
	Therefore 
	\begin{eqnarray*}
		\Re(b\sqrt{b} z^{R}Z_\frac1{b}B_2(x,\nu)) =
		b-x -2b + (R+1) = -x+ 1-F. 
	\end{eqnarray*}
	Since $x \in [0, F]$, the above equality is nonzero. Therefore, on $\itvcc{0,F}\times \itvoo{0, 1}$, the Zak transform of $B_2$ is zero if and only if 
    $$ x= F, \quad\text{ and, }\quad \nu= \frac{s}{R},\text{ for some } s=1,\dots, R -1. $$

\emph{Case 2:} For $x\in\itvoc{F, \frac12}$, we have 
$\lceil -b + x \rceil = -R+1$ and   $ \lfloor   b+ x \rfloor = R.$
Therefore
\begin{align}\label{eq:2307a}
\sqrt{b}Z_{\frac{1}{b}}B_2(x,\nu)
&=
\sum_{k=-R+1}^{0}
\left(1-\frac{x-k}{b}\right)\myexp{2\pi i k\nu}
+
\sum_{k=1}^{R}
\left(1+\frac{x-k}{b}\right)\myexp{2\pi i k\nu}
\notag\\
&=
\frac{1-z^{-R}}{b(1-z^{-1})^2}
\left[
(1-z^{-1})\bigl((b+x-R)z^R+b-x-R+1\bigr)
+z^{R-1}-1
\right].
\end{align}
If $z^R=1$, that means, if $\nu=\frac{s}{R}$ for some $s=1,\dots,R-1$, then
\[
Z_{\frac{1}{b}}B_2(x,\nu)=0,
\qquad
\text{for all } x\in\itvoc{F,1/2}.
\]
Alternatively, the Zak transform is zero if the term in the bracket in 
\eqref{eq:2307a} vanishes, that is,
\begin{align*}
(1-z^{-1})\bigl((b+x-R)z^R+b-x-R+1\bigr)
+z^{R-1}-1 = 0.
\end{align*}
Solving this equation for $x$ gives 
\begin{align*}
x
&=
-\frac{z^{-1}}{1-z^{-1}}
-2F\frac{z^{-R}}{1-z^{-R}}
-F \\
&=
\frac12
+i\left(
\frac12\cot(\pi\nu)
+
F\cot(\pi R\nu)
\right).
\end{align*}
Since $x$ is real, the imaginary part must vanish. Hence, the remaining zeros of the Zak transform occur at $x=\frac12$ and $v\in \itvoo{0, 1}$ satisfying
\[
\cot(\pi\nu)+2F\cot(\pi R\nu)=0.
\]
\end{proof}

% \begin{proposition}[{[[DELETE?]]}]
% 	Let $1< \lambda <2$. We consider $(x,\nu) \in \itvcc{0,1/2}^2$. Then 
% 	 $$Z_{\lambda}B_2(x,\nu)=0 \Longleftrightarrow x = \nu =\frac12.
% 	$$
% \end{proposition}

% \begin{proof} The result follows from the explicit formula
% \begin{align*}
% 	 \frac{1}{\sqrt{\lambda} }Z_{\lambda}B_2(x, \nu) = 
%      \left\{\begin{array}{lll}
%      \frac1\lambda -x,&&0 \leq x \leq F\\
%      \frac1\lambda-x+(\frac1\lambda-1+x)\myexp{2 \pi i \nu},&&F<x\leq 1/2.
%      \end{array}\right.
% \end{align*}
% \end{proof}

\section{Some observations on $Z_\lambda B_n$ along the line $x=1/2$}
\label{sec:palindromic-polynomials}

Along the line $x=1/2$, the Zak transform has an additional symmetry. Indeed,
with $z=\myexp{2\pi i\nu}$,
\[
Z_\lambda B_n(1/2,\nu)
=
\sqrt{\lambda}
\sum_{k\in\mathbb Z}
B_n\bigl(\lambda(\tfrac12-k)\bigr)z^k,
\]
and the coefficients satisfy
\[
B_n\bigl(\lambda(\tfrac12-k)\bigr)
=
B_n\bigl(\lambda(\tfrac12-(1-k))\bigr). 
\]
Thus, after multiplication by a suitable power of $z$, the function
$Z_\lambda B_n(1/2,\nu)$ is represented by a palindromic polynomial.
Consequently, its zeros for real $\nu$ correspond to roots of a
palindromic polynomial on the unit circle. Thus, if $z$ is a zero, then so is $1/z$. In particular, if $z$ is a zero on the unit circle, then $z$ and $\overline z$ are both zeros. Hence, the zeros of $Z_\lambda B_n(1/2,\cdot)$ for real $\nu$ are symmetric with respect to $\nu=0$ and $\nu=1/2$.

For $n=2$, these zeros were described explicitly in \Cref{thm:char-zak-B2-lambda-le-1}. Write
$1/\lambda=R+F, R=\round{1/\lambda}, F=\sfrac{1/\lambda}$.
Besides the zeros
\[
\nu=\frac{s}{R},
\qquad
s=1,\ldots,R-1,
\]
the remaining zeros are determined, away from these points, by the
cotangent equation $\cot(\pi v) + 2F \cot(\pi R v)=0$. We remark that the cotangent equation has exactly $R$ solutions in $\itvoo{0,1}$, and these solutions are interlaced with the points $\frac{s}{R}$ for $s=0,1\dots, R-1, R$ (see \Cref{thm:zeros-zak-bspline}). The proof of this fact is a straightforward, and we leave it to the reader.
As $F$ varies in the rounding interval, the solutions of this equation
interpolate between the neighboring rational grids. Thus, for
$B_2$, the two families of zeros are naturally linked.

For higher-order B-splines the corresponding equations become more
complicated. The linear cotangent condition for $B_2$ is replaced by
higher-order conditions involving polynomials in $\cot(\pi\nu)$ of
degree up to $n-1$. 

The case $n=3$ already illustrates the difference. Suppose
\[
\frac1\lambda=R+F,
\qquad
R \text{ even},
\qquad
0<\abs F<\frac13.
\]
Writing
\[
t=\cot(\pi\nu),\qquad
\sigma=\sin(R\pi\nu),\qquad
c=\cos(R\pi\nu),
\]
the zero condition can be written, up to a non-zero factor, as
\[
\sigma\left(
\sigma^2(2t^2+1)
+
6F\sigma c\,t
+
3F^2(2-3\sigma^2)
\right)=0.
\]
The quadratic factor has no real zeros for $\abs F<1/3$. Hence, in this
case, the only zeros for real $\nu$ are
$\nu=s/R,s=1,\ldots,R-1$. This example indicates that the simple interpolation picture for
$B_2$ does not extend directly to higher orders.

% \section{Palindromic polynomials and the zeros of $Z_\lambda B_n(1/2, \cdot)$}
% \label{sec:palindromic-polynomials}

% For fixed $x$ the Zak transform $Z_\lambda B_n(x,\cdot)$ is a trigonometric polynomial of degree at most $2\round{n/\lambda}+1$ hence the number of zeros of $Z_\lambda B_n(x,\cdot)$ is at most $2\round{n/\lambda}+1$ for each fixed $x$. Proposition \ref{thm:zeros-zak-bspline} shows that for $\lambda < 2/3$ the Zak transform $Z_\lambda B_n(x,\cdot)$ may have $\round{1/\lambda}-1$ zeros for each $x \in \itvcc*{\frac{n}{2} \abs{\sfrac{1/\lambda}},1-\frac{n}{2} \abs{\sfrac{1/\lambda}}}$.

% The Zak transform $Z_\lambda B_n$ along the line $x=1/2$ is a palindromic polynomial in the variable $z = \myexp{2\pi i \nu}$.
% Palindromic polynomials are polynomials $P(z) = \sum_{k=0}^m a_k z^k$ such that $a_k = a_{m-k}$ for all $k$. The zeros of palindromic polynomials have a special structure, and we will show that the zeros of $Z_\lambda B_n(1/2, \cdot)$ are precisely the roots of a certain class of palindromic polynomials.

%     [[WHAT TO ADD HERE?]]

\section{The frame property of integer-oversampled Gabor systems}
\label{sec:frame-property-gabor-systems}

From \Cref{cor:zero-set-Zak-B_n-for-lambda-ge-1} it follows that the  integer-oversampled Gabor systems $\gaborG{B_n}$, i.e., $ab=1/q$, generated by $B_n$ are frames for all $b \le 1$ and $a<n$.

\begin{theorem}
    \label{thm:frame-small-b}
    Let $n \ge 2$ and $q \ge 2$ be integers, and let $a>0$. Assume $0<b\le 1$. Let $q \ge 2$, and let $ab=1/q$. The Gabor system $\mathcal{G}(B_n,a,1/(qa))$ is a frame if and only if $a<n$.
\end{theorem}
\begin{proof}
    Recall that $\gaborG{B_n}$ fails to be a frame precisely when there
    exist $q$ zeros of $Z_{1/b}B_n$ on the same horizontal line whose
    $x$-coordinates are $1/q$ apart.

    Suppose first that $\frac1n<b\le1$.
    By \Cref{cor:zero-set-Zak-B_n-for-lambda-ge-1},
    $Z_{1/b}B_n$ has the single zero $(1/2,1/2)$ in
    $\itvco{0,1}^2$. Hence no such set of $q$ zeros exists, and
    $\gaborG{B_n}$ is a frame. Notice that in this case $a=\frac1{qb}<n$
    automatically.

    Now suppose that $0<b\le\frac1n$.
    Again by \Cref{cor:zero-set-Zak-B_n-for-lambda-ge-1}, for every
    fixed $\nu$ the nonzero set of $Z_{1/b}B_n(\cdot,\nu)$, regarded
    on $\R/\Z$, is an interval of length $nb$. Thus $q$ points
    spaced $1/q$ apart can all lie in the zero set if and only if
    \[
        nb\le\frac1q.
    \]
    Since $ab=1/q$, this is equivalent to $a\ge n$.
    Hence $\gaborG{B_n}$ is a frame if and only if $a<n$.
\end{proof}

From \Cref{cor:zero-set-Zak-B_n-for-lambda-ge-1} combined with \Cref{thm:char-zak-B2-lambda-le-1}, we in a similar way obtain a complete characterization of the frame property of integer-oversampled Gabor systems generated by the hat spline $B_2$.
\begin{theorem}
    \label{thm:frame-characterization}
    Let $b > 0$ and $q \ge 2$. The integer-oversampled Gabor system $\mathcal{G}(B_2,1/(qb),b)$ is a frame if and only if, we have $\round{b} = 1$ \emph{or} $\abs{\sfrac{b}} > 1/(2q)$ when $\round{b}=0,2,3,4,\dots$.
\end{theorem}

\bibliographystyle{abbrv}
\bibliography{jl_bib}

\end{document}